\documentclass[11pt,reqno]{amsart}
\usepackage{indentfirst, amssymb,amsmath,amsthm, mathrsfs,cite,graphicx,float}
\newtheorem*{theoA}{Theorem A}

\newtheorem{theo}{Theorem}[section]
\newtheorem{lem}{Lemma}[section]

\newtheorem{exam}{Example}[section]

\newtheorem{ques}{Question}[section]

\newtheorem{open problem}{Open problem}[section]
\newcommand{\pa}{\partial}
\newcommand{\ol}{\overline}
\newcommand{\be}{\begin{equation}}
\newcommand{\ee}{\end{equation}}
\newcommand{\bs}{\begin{small}}
\newcommand{\es}{\end{small}}
\newcommand{\beas}{\begin{eqnarray*}}
\newcommand{\eeas}{\end{eqnarray*}}
\newcommand{\bea}{\begin{eqnarray}}
\newcommand{\eea}{\end{eqnarray}}
\renewcommand{\epsilon}{\varepsilon}
\numberwithin{equation}{section}

\begin{document}
\title[Bohr inequalities]
{Bohr inequalities for holomorphic mappings in higher-dimensional complex Banach spaces }
\author[V. Allu, R. Biswas and R. Mandal]{Vasudevarao Allu, Raju Biswas and Rajib Mandal}
\date{}
\address{Vasudevarao Allu, Department of Mathematics, School of Basic Science, Indian Insti
tute of Technology Bhubaneswar, Bhubaneswar-752050, Odisha, India.}
\email{avrao@iitbbs.ac.in}
\address{Raju Biswas, Department of Mathematics, Raiganj University, Raiganj, West Bengal-733134, India.}
\email{rajubiswasjanu02@gmail.com}
\address{Rajib Mandal, Department of Mathematics, Raiganj University, Raiganj, West Bengal-733134, India.}
\email{rajibmathresearch@gmail.com}
\let\thefootnote\relax
\begin{abstract} In this paper, we investigates the Bohr phenomenon for holomorphic mappings $F$ from the unit ball $\mathbb{B}_X$ of a complex Banach space $X$ into 
the closure of the unit polydisc $\mathbb{D}^m$ within the space $\mathbb{C}^m$. First, we prove an improved Bohr inequality involving the squared norms of the mapping and its homogeneous expansions. Second, we derive a refined Bohr inequality that incorporates a combination of the coefficient norms and their squares. Finally, we obtain a refined Bohr inequality for compositions $F\circ \nu$, where $\nu$ is a Schwarz mapping with a zero of order $k$ at the origin. For each result, we demonstrate that the derived Bohr radius is sharp. 
\end{abstract}
\maketitle
\section{ Introduction and preliminaries}
\noindent Let $f(z)=\sum_{n=0}^{\infty} a_nz^n$ be a holomorphic function in the unit disk $\mathbb{D}:=\{z\in\mathbb{C}:|z|<1\}$. The classical Bohr's theorem states that if 
$|f(z)|\leq 1$ in $\Bbb{D}$, then 
\be\label{e1}\sum_{n=0}^{\infty}| a_n| r^n\leq 1\quad\text{for}\quad |z|=r\leq 1/3.\ee
In this context, the number $1/3$ is known as the Bohr radius and cannot be further improved (see \cite{B1914}). The inequality (\ref{e1}) is known as the classical Bohr inequality. 
Bohr's \cite{B1914} original proof of the inequality (\ref{e1}) was restricted to situations where $r$ did not exceed $1/6$. Later, Schur, Riesz and Wiener independently proved that the inequality (\ref{e1}) is valid for $r\leq 1/3$ (see \cite{B1914}).\\[2mm]
\indent  In addition to the Bohr radius, the Rogosinski radius is another significant concept (see \cite{R1923}). Kayumov {\it et al.} \cite{KKP2021} have defined the Bohr-Rogosinski radius as the greatest value $r_0 $ in $(0,1)$ such that the inequality $R_M^f(z) \leq 1$ is valid for all points within the disk $|z|\leq r_0$, where
 \beas R_M^f(z):=|f(z)|+\sum_{n=M}^\infty |a_n| |z|^n,\quad M\in\mathbb{N}\eeas
 for a holomorphic function $f(z)=\sum_{n=0}^\infty a_n z^n$ in $\mathbb{D}$ with $|f(z)|\leq 1$ in $\mathbb{D}$. Furthermore, the authors \cite{KKP2021} derived the following result concerning the Bohr-Rogosinski radius.
\begin{theoA}\cite[Theorems 1 and 2, p. 3]{KKP2021}
Let $f(z)=\sum_{n=0}^\infty a_n z^n$ be a holomorphic mapping in $\mathbb{D}$ with $|f(z)|\leq 1$ in $\mathbb{D}$. For $k,M\in\mathbb{N}$ and $p\in\{1,2\}$, we have 
\beas \left|f\left(z^k\right)\right|^p+\sum_{n=M}^\infty |a_n| |z|^n\leq 1 \quad\text{for}\quad |z|\leq r_{k,p,M},\eeas
where $r_{k,p,M}\in (0, 1)$ is the unique root of the equation $(2/p)r^M (1+r^k)-(1-r)(1-r^k)=0$. The number $r_{k,p,M}$ cannot be improved.
\end{theoA}
\noindent Recently, there has been a notable increase in research focusing on Bohr inequalities in various abstract settings, such as ramification, refinement, and extension. For an 
in-depth study of the Bohr radius in one complex variable, we refer to 
\cite{KKP2020, AKP2019,AA2023,BB2004,1KP2018,KP2018, LP2023, ABM2025, BM2024,PVW2020, PW2020,1AH2022, KP2017, IKP2020, KPS2018} and the references therein.\\[2mm]
\indent To organize and present our results coherently, it is necessary to introduce basic notations at this preliminary stage. For $z=(z_1,z_2,\cdots,z_n)\in\mathbb{C}^n$, let $\Vert z\Vert_\infty=\max_{i=1,2,\cdots,n}|z_i|$ denote the maximum norm and let
\beas \mathbb{D}^n =\left\{z=(z_1,z_2,\cdots,z_n)\in\mathbb{C}^n: \Vert z\Vert_\infty < 1\right\}\eeas
denote the unit polydisc in $\mathbb{C}^n$. Let $X$ and $Y$ be two complex Banach spaces. A mapping $P: X\to Y$ is called a homogeneous polynomial of degree $k\in\mathbb{N}$ if there exists a $k$-linear mapping 
$f :\underbrace{X\times X\times\ldots\times X}_{k\text{-times}}\to Y$ such that $P(x)=f(x,x,\ldots,x)$ for every $x\in X$. Note that if $P$ is a homogeneous polynomial of degree $k$, then $P(\mu x)=\mu^k P(x)$ for all $x\in X$ and $\mu\in \mathbb{C}$. Note that if $P_k$ is a $k$-homogeneous polynomial from $X$ into $Y$, there exists a unique symmetric $k$-linear mapping $f$ such that $P_k(x)=f(x,...,x)$ (see \cite[Chapter 1 and 2]{M1986}). \\[2mm]
\indent For a domain $\mathcal{D}\subset X$, let $F: \mathcal{D}\to Y$ be a holomorphic mapping and let $D^kF(x)$ denote the $k$-th Fr\'echet derivative of $F$ at $x\in\mathcal{D}$, which is a bounded symmetric $k$-linear mapping from $\underbrace{X\times X\times\ldots\times X}_{k\text{-times}}$ to $Y$. If $\mathcal{D}$ contains the origin, then
 any holomorphic mapping $F : \mathcal{D} \to Y$ can be expanded into the following series:
 \bea\label{e2} F(z) =\sum_{k=0}^\infty \frac{1}{k!} D^kF(0)(z^k)\eea
 for all $z$ in some neighborhood of the origin (see \cite[Chapter 6, p. 198]{GK2003}). Note that 
 \beas  D^0 F(0)(z^0)=F(0)\quad\text{and} \quad D^kF(0)(z^k)= D^kF(0)\underbrace{(z,z,\ldots,z)}_{k\text{-times}},\quad k\geq 1.\eeas 
 As $(1/k!)D^kF(0)(z^k)$ is a bounded symmetric $k$-linear mapping, we use the notation $P_k(z)=(1/k!)D^kF(0)(z^k)$ throughout the paper. 
 In the context of a balanced domain $\mathcal{D}$, the series (\ref{e2}) converges uniformly to $f$ in some neighborhood of each compact subset contained in 
 $\mathcal{D}$. Furthermore, if the function $f$ is bounded on the ball $B(z, r)\subseteq \mathcal{D}$, then uniform convergence is guaranteed on $B(z, t)$ for any $0 < t < r$ (see \cite[Theorems 7.11 and 7.13]{M1986}).\\[2mm]
\indent  Let $L(X, \mathbb{C})$ denote the space of all continuous linear functionals from $X$ to $\mathbb{C}$. For each $x\in X \setminus\{0\}$, let
 \beas T(x)=\left\{l_x \in L(X, \mathbb{C}): l_x(x)=\Vert x\Vert, ~\Vert l_x\Vert =1\right\},\eeas
 where $\Vert l_x\Vert=\sup\{|l_x(w)| : \Vert w\Vert=1\}$.
In view of the Hahn-Banach theorem, $T(x)$ is nonempty (see \cite[p. 202]{GK2003}). \\[2mm]
\indent  In 2018, Kayumov and Ponnusamy \cite{KP2018} obtained several Bohr inequalities for holomorphic functions with lacunary series in the unit disc $\mathbb{D}$.  The study of Bohr's phenomenon in the context of higher-dimensional complex Banach spaces represents a significant advancement in the field of complex analysis.
 In 2020, Liu and Liu \cite{LL2020} made significant contributions to the field by extending the results of Kayumov and Ponnusamy \cite{KP2018} to holomorphic mappings $f$ with 
 lacunary series from the unit polydisc $\mathbb{D}^m$ into $\mathbb{D}^m$ and from the unit ball $\Bbb{B}_X$ of a complex Banach space $X$ into $\Bbb{B}_X$. 
For an in-depth study of Bohr's phenomenon concerning holomorphic or pluriharmonic mappings with values in higher-dimensional complex Banach spaces using the homogeneous polynomial expansions, we refer to
 \cite{H2023, 1HH2024, KPW2025, A2000, HHK2009, HHK2025}.
Lin {\it et al.} \cite{LLP2023} have studied other types of Bohr's inequality for holomorphic mappings $f$ with lacunary series from the unit polydisc $\mathbb{D}^m$ into $\mathbb{D}^m$.\\[2mm]
\indent The aforementioned studies primarily explore mappings between domains of the same type, such as from a ball to another ball or from a polydisc to another polydisc, using their inherent symmetries. The case of mappings from a Banach space ball into a polydisc, {\it i.e.,} $F:\mathbb{B}_X\to\ol{\mathbb{D}^m}$, presents a distinct challenge due to the structural asymmetry between the domain and the codomain. This setting disrupts the standard techniques and introduces a critical dependence on the geometry of the initial point $F(0)=a$.
This paper addresses this gap by providing a comprehensive study of the Bohr phenomenon in this general and asymmetric setting.
The following question naturally arises.
 \begin{ques}\label{Q1}Let $\mathbb{B}_X$ be the unit ball of a complex Banach space $X$.
 To what extent can the classical Bohr inequality be generalized to derive sharp versions for holomorphic mappings $F:\mathbb{B}_X\to\ol{\mathbb{D}^m}$, and how does the geometry of the initial point $F(0)$ affect the associated Bohr radius?
 \end{ques}
  \noindent The purpose of this paper is to provide an affirmative answer to Question \ref{Q1}.
\section{Key lemmas}
\noindent The following are key lemmas of this paper and will be used to prove the main results.
\begin{lem}\cite{DP2008}\label{lem1} Suppose $f$ is analytic in $\mathbb{D}$ with $|f(z)|\leq1$ in $\Bbb{D}$, then we have
\beas \frac{\left|f^{(n)}(z)\right|}{n!}\leq \frac{1-|f(z)|^2}{(1-|z|)^{n-1}(1-|z|^2)}\quad\text{and}\quad |a_n|\leq 1-|a_0|^2\quad\text{for}\quad n\geq 1,\quad |z|<1.\eeas\end{lem}
\begin{lem}\cite{KP2017}\label{lem2} If $f(z)=\sum_{n=0}^\infty a_n z^n$ is analytic in $\mathbb{D}$ such that $|g(z)|\leq 1$ in $\mathbb{D}$ and $|a_0|<1$, then we have
\beas \sum_{n=1}^\infty |a_n|^2 r^{pn}\leq r^p\frac{(1-|a_0|^2)^2}{1-|a_0|^2 r^p},\quad \text{where}\quad p\in\mathbb{N}\quad\text{and}\quad 0<r\leq 1. \eeas 
\end{lem}
\begin{lem}\cite{PVW2020}\label{2lem2} If $f(z)=\sum_{n=0}^\infty a_n z^n$ is analytic in $\mathbb{D}$ such that $|g(z)|\leq 1$ in $\mathbb{D}$ and $|a_0|<1$, then we have
\beas \sum_{n=1}^\infty |a_n| r^{n}+\left(\frac{1}{1+|a_0|}+\frac{r}{1-r}\right) \sum_{n=1}^\infty |a_n|^2 r^{2n}\leq (1-|a_0|^2)\frac{r}{1-r},\quad \text{where}\quad 0\leq r<1. \eeas 
\end{lem}
\begin{lem}\cite{CHPV2024}\label{lem3} Suppose that $\Bbb{B}_X$ and $\Bbb{B}_Y$ are the unit balls of the complex Banach spaces $X$ and $Y,$ respectively. Let $f: \Bbb{B}_X\to\ol{\Bbb{B}_Y}$ be a holomorphic mapping. Then,
\beas \Vert f(z)\Vert_Y\leq \frac{\Vert f(0)\Vert_Y+\Vert z\Vert_X}{1+\Vert f(0)\Vert_Y\Vert z\Vert_X}\quad \text{for}\quad z\in \Bbb{B}_X. \eeas  
This estimate is sharp with equality possible for each value of $\Vert f(0)\Vert_Y$ and for each $z\in \Bbb{B}_X$.
\end{lem}
\begin{lem}\label{lem4}\cite[Lemma 6.1.28, P. 200]{GK2003} Let $X$ and $Y$ be two complex Banach spaces.
Let $M>0$, $\Bbb{B}_X$ be the unit ball in the complex Banach space $X$ and $f: \Bbb{B}_X\to Y$ be a holomorphic mapping such that $f(0)=Df(0)=\cdots=D^{k-1}f(0)=0$, $k\in\mathbb{N}$ and $\Vert f(z)\Vert_Y <M$ for $z\in \Bbb{B}_X$. Then, $\Vert f(z)\Vert_Y \leq M\Vert z\Vert_X^k$ for $z\in \Bbb{B}_X$. 
\end{lem}
\section{Main results}
\noindent In the following result, we establish an improved version of Bohr inequality for holomorphic mappings $F:\mathbb{B}_X\to\ol{\mathbb{D}^m}$, where $\mathbb{B}_X$ is the unit ball of a complex Banach space $X$.
\begin{theo}\label{Th1}
Let $a=(a_1, a_2, \ldots, a_m)\in\ol{\mathbb{D}^m}$ and let $F: \mathbb{B}_X\to \ol{\mathbb{D}^m}$ be a holomorphic mapping with the following series expansion: 
\beas F(z)=a+\sum_{n=1}^\infty P_n(z)\quad\text{for}\quad z\in \mathbb{B}_X,\eeas
where $P_n(z)=(1/n!) D^n F(0)(z^n)$, $n\geq 1$. Assume that $|a_i|=\Vert a\Vert_\infty$ for $1\leq i\leq m$. Then, we have 
\bea\label{a1}\Vert F(z)\Vert_\infty^2+\sum_{n=1}^\infty \Vert P_n(z)\Vert_\infty^2\leq 1\eea
for $\Vert z\Vert_X=r\leq \sqrt{11/27}$. The number $ \sqrt{11/27}$ cannot be improved.
 \end{theo}
\begin{proof} Let $z\in\mathbb{B}_X\setminus\{0\}$ be fixed and denote $z_0=z/\Vert z\Vert_X$. Then, $z_0\in\pa \mathbb{B}_X$ and $z=r z_0$, where $\Vert z\Vert_X=r$. 
As $F(z)=(F_1(z),F_2(z),\ldots,F_m(z))$, we consider the function $g_i: \mathbb{D}\to\ol{\mathbb{D}}$ defined by 
\beas g_i(t):=F_i(t z_0)=a_i+\sum_{n=1}^\infty (P_{n})_i(z_0)t^n,\quad t\in\mathbb{D}\quad\text{and}\quad i\in\{1,2,\ldots,m\}.\eeas
The condition $\Vert a\Vert_\infty=|a_j|$ for each $j\in\{1,2,\ldots, m\}$ ensures that the maximum modulus of the initial value is attained in every component, which is essential for obtaining the uniform estimates by maximizing over the component index $j$.
In view of \textrm{Lemma \ref{lem2}}, we have
\bea\label{a3} \sum_{n=1}^\infty |(P_{n})_i(z_0)|^2 |t|^{2n}\leq |t|^2\frac{\left(1-|a_i|^2\right)^2}{1-|a_i|^2 |t|^2}=|t|^2\frac{\left(1-\Vert a\Vert_\infty^2\right)^2}{1-\Vert a\Vert_\infty^2 |t|^2},\quad |t|<1\eea
for each fixed $i\in\{1,2,\ldots,m\}$.
Let $k\in\{1,2,\ldots,m\}$ be such that $|(P_{n})_k(z_0)|=\Vert P_n(z_0)\Vert_\infty=\max_{1\leq i\leq m} \left|(P_n)_i(z_0)\right|$. By setting $|t|=r<1$ in (\ref{a3}), we have 
\bea\label{a2}\sum_{n=1}^\infty \Vert P_n(z_0)\Vert_\infty^2 r^{2n}=\sum_{n=1}^\infty \left(\max_{1\leq i\leq m}\left\{|(P_{n})_i(z_0)|\right\}\right)^2 r^{2n}&=&\sum_{n=1}^\infty|(P_{n})_k(z_0)|^2 r^{2n}\nonumber\\&\leq& r^2\frac{\left(1-\Vert a\Vert_\infty^2\right)^2}{1-\Vert a\Vert_\infty^2 r^2}.\eea
Let $x=\Vert a\Vert_{\infty}\in[0,1]$. Evidently, if $x=1$, then we have $P_n(z_0)=0$ for $n\geq 1$. Thus, the inequality (\ref{a1}) holds for $z=rz_0$ and $r\in[0,1)$. 
Let $x\in[0,1)$. Using (\ref{a2}), for $r\in(0,1)$, we have
\beas \Vert F(r z_0)\Vert_\infty^2+\sum_{n=1}^\infty \Vert P_n(r z_0)\Vert_\infty^2&\leq&\left(\frac{x+r}{1+x r}\right)^2+r^2\frac{\left(1-x^2\right)^2}{1-x^2 r^2}\\[2mm]
&=&1+G_1(x, r),\eeas
where
\beas G_1(x,r)=\left(\frac{x+r}{1+x r}\right)^2+r^2\frac{\left(1-x^2\right)^2}{1-x^2 r^2}-1,  \quad x\in[0,1)\quad\text{and}\quad r\in[0,1).\eeas
To establish our result, it suffices to demonstrate that $G_1(x,r)\leq 0$.
Differentiating partially $G_1(x,r)$ with respect to $r$, we obtain
\beas \frac{\pa}{\pa r}G_1(x,r)&=&2\left(\frac{x+r}{1+x r}\right)\left(\frac{(1+x r)-(x+r) x}{(1+x r)^2}\right)\\[2mm]
&&+\frac{2 r \left(1-x^2\right)^2}{1-r^2 x^2}+\frac{2 r^3 x^2 \left(1-x^2\right)^2}{\left(1-r^2 x^2\right)^2}\\[2mm]
&=&\frac{2 \left(1-x^2\right) \left(x+r\right)}{\left(1+x r\right)^3}+\frac{2 r \left(1-x^2\right)^2}{1-r^2 x^2}+\frac{2 r^3 x^2 \left(1-x^2\right)^2}{\left(1-r^2 x^2\right)^2}\geq 0,\eeas
which shows that $G_1(x,r)$ is a monotonically decreasing function of $r\in[0, 1)$ and it follows that for $r\leq \sqrt{11/27}$, we have
\beas G_1(x,r)&\leq& G_1\left(x,\sqrt{11/27}\right)\\[2mm]
&=&\frac{3 \left(121 x^6+66 \sqrt{33} x^5-121 x^4-132 \sqrt{33} x^3+135 x^2+66 \sqrt{33} x-135\right)}{\left(\sqrt{33} x+9\right)^2 \left(27-11 x^2\right)}.\eeas
Let $G_2(x):=121 x^6+66 \sqrt{33} x^5-121 x^4-132 \sqrt{33} x^3+135 x^2+66 \sqrt{33} x-135$. Following a basic computation, we obtain
\beas G_2(x)=121(x^2-1) \left(x+5 \sqrt{\frac{3}{11}}\right) \left(x+3 \sqrt{\frac{3}{11}}\right) \left(x-\sqrt{\frac{3}{11}}\right)^2\leq 0,\quad x\in[0, 1)\eeas
and $G_2(\sqrt{3/11})=0$. Therefore, $G_1(x, r)\leq 0$ for $r\leq \sqrt{11/27}$.\\[2mm]
\indent To prove the sharpness of the constant $\sqrt{11/27}$, we consider the function $F_1: \mathbb{B}_X\to \ol{\mathbb{D}^m}$ defined by
\beas F_1(z)=g(l_{z_0}(z))u\quad\quad\text{for}\quad z\in \mathbb{B}_X,\eeas
where $z_0\in \pa \mathbb{B}_X$, $l_{z_0}\in T(z_0)$, $u=(u_1,u_2,\ldots,u_m)\in \pa \mathbb{D}^m$ with $|u_1|=|u_2|=\dots =|u_m|=1$, and 
\beas g(t)= \frac{\lambda +t}{1+\lambda t},\quad  t \in \mathbb{D}\quad \text{and}\quad \lambda=\sqrt{3/11}.\eeas
Therefore, for $r\in(0, 1)$, we have 
\beas F_1(rz_0)=g(l_{z_0}(rz_0))u =g\left(r\Vert z_0\Vert\right)u=\frac{\lambda +r}{1+\lambda r}u= a+\sum_{n=1}^\infty P_n(rz_0),\eeas
where $a=\lambda u$ and $P_n(z_0)=(-1)^n \lambda^{n-1}(\lambda^2 -1)u$ for $n\geq 1$. 
Thus, we have
\beas \Vert F_1(r z_0)\Vert_\infty^2+\sum_{n=1}^\infty \Vert P_n(r z_0)\Vert_\infty^2&=&\left(\frac{\lambda +r}{1+\lambda r}\right)^2+(1-\lambda^2)^2\sum_{n=1}^\infty \lambda^{2(n-1)}r^{2n}\nonumber\\
&=&\left(\frac{\lambda +r}{1+\lambda r}\right)^2+(1-\lambda^2)^2\frac{r^2}{1-\lambda^2 r^2}\\
&=&\frac{64 r^2}{121-33 r^2}+\frac{\left(11 r+\sqrt{33}\right)^2}{\left(\sqrt{33} r+11\right)^2}>1\quad\text{for}\quad r>\sqrt{\frac{11}{27}},\eeas
which shows that the constant $\sqrt{11/27}$ cannot be improved. This completes the proof.
\end{proof}
\noindent 
The condition $\Vert a\Vert_\infty=|a_j|$ for each $j\in\{1,2,\ldots, m\}$ is vital in \textrm{Theorem \ref{Th1}} and cannot be omitted. To illustrate the importance of this condition, we present the following example.
\begin{exam}\label{Exam1} Consider the holomorphic mapping $F:\mathbb{B}_X\to\ol{\mathbb{D}^2}$ with
\[F(z)=\left(\frac{a_1+l_{z_0}(z)}{1+a_1 l_{z_0}(z)}, \frac{a_2+l_{z_0}(z)}{1+a_2 l_{z_0}(z)}\right),~z\in\mathbb{B}_X, \] 
where $a_1\in (0,1)$, $a_2\in (0,1)$, $z_0\in\pa \mathbb{B}_X$ and $l_{z_0}\in T(z_0)$. For $r\in(0, 1)$, we have
\beas F(r z_0)=\left(\frac{a_1+l_{z_0}(r z_0)}{1+a_1 l_{z_0}(r z_0)}, \frac{a_2+l_{z_0}(r z_0)}{1+a_2 l_{z_0}(r z_0)}\right)=\left(\frac{a_1+r}{1+a_1 r}, \frac{a_2+r}{1+a_2 r}\right)=a+\sum_{n=1}^\infty P_n(r z_0),\eeas
where $a=(a_1, a_2)$ and $P_n(r z_0)=\left((-1)^n (a_1^2-1)a_1^{n-1}r^n, (-1)^n(a_2^2-1)a_2^{n-1}r^n\right)$ for $n\geq 1$. 
Note that
\beas \Vert P_n(r z_0)\Vert_\infty=\max\left\{\left|(-1)^n (a_1^2-1)a_1^{n-1}r^n\right|, \left|(-1)^n(a_2^2-1)a_2^{n-1}r^n\right|\right\}\quad\text{for}\quad n\geq 1.\eeas
Let $H_k(t):=(1-t^2)t^{k-1}$, where $t\in(0, 1)$ and $k\;(\geq 1)$. Evidently,
\[H_k'(t)=(k-1) t^{k-2}-(k+1)t^k=(k-1) t^{k-2}\left(1-\frac{k+1}{k-1}t^2\right)\leq 0\quad\text{for}\quad t^2\geq \frac{k-1}{k+1}.\]
Thus, $H_k(t)$ is a monotonically decreasing function of $t$ when $t\geq \sqrt{(k-1)/(k+1)}$.
Suppose that $\sqrt{1/3}<a_1<a_2<1$. 
Hence, we have $\Vert P_n(r z_0)\Vert_\infty=(1-a_1^2)a_1^{n-1}r^n $ for $n=1,2$ and $\Vert P_n(r z_0)\Vert_\infty\geq (1-a_2^2)a_2^{n-1}r^n$ for $n\geq 3$.
Therefore, we have
\beas&& \Vert F(r z_0)\Vert_\infty^2+\sum_{n=1}^\infty \Vert P_n(r z_0)\Vert_\infty^2\\
&=&\left(\frac{a_2 +r}{1+a_2 r}\right)^2+(1-a_1^2)^2r^2(1+a_1^2r^2)+\sum_{n=3}^\infty \Vert P_n(r z_0)\Vert_\infty^2\\
&\geq &\left(\frac{a_2 +r}{1+a_2 r}\right)^2+(1-a_1^2)^2r^2(1+a_1^2r^2)+(1-a_2^2)^2\sum_{n=3}^\infty a_2^{2(n-1)}r^{2n}\\
&=&\left(\frac{a_2 +r}{1+a_2 r}\right)^2+(1-a_1^2)^2r^2(1+a_1^2r^2)+(1-a_2^2)^2\frac{a_2^4 r^6}{1-a_2^2 r^2}\eeas
Thus, for each $r\in(0, 1)$, we have  
\beas \Vert F(r z_0)\Vert_\infty^2+\sum_{n=1}^\infty \Vert P_n(r z_0)\Vert_\infty^2\geq 1+(1-a_1^2)^2r^2(1+a_1^2r^2)>1\quad\text{as}\quad a_2\to 1^-.\eeas
\end{exam}
In the following theorem, we establish a refined version of the Bohr inequality for holomorphic mapping $F:\mathbb{B}_X\to\ol{\mathbb{D}^m}$, where $\mathbb{B}_X$ is the unit ball of a complex Banach space $X$.
\begin{theo}\label{Th2}
Let $a=(a_1, a_2, \ldots, a_m)\in\ol{\mathbb{D}^m}$, $p\in\{1,2\}$ and let $F: \mathbb{B}_X\to \ol{\mathbb{D}^m}$ be a holomorphic mapping with the following expansion:
\beas F(z)=a+\sum_{n=1}^\infty P_n(z),\quad z\in \mathbb{B}_X,\eeas
where $P_n(z)=(1/n!) D^n F(0)(z^n)$. Assume that $|a_i|=\Vert a\Vert_\infty$ for all $i\in\{1, 2, \ldots,m\}$. Then, we have 
\be\label{c1}
\Vert F(z)-F(0)\Vert_\infty+\Vert F(0)\Vert_\infty^p+\sum_{n=1}^\infty \Vert P_n(z)\Vert_\infty+\left(\frac{1}{1+\Vert a\Vert_\infty}+\frac{\Vert z\Vert_X}{1-\Vert z\Vert_X}\right)\sum_{n=1}^\infty \Vert P_n(z)\Vert_\infty^2\leq 1\ee
for $\Vert z\Vert_X=r\leq r_p$ and the number $r_p$ cannot be improved, where $r_1=1/5$ and $r_2=1/3$.
 \end{theo}
\begin{proof}
Let $z_0\in\pa \mathbb{B}_X$ be fixed. As $F(z)=(F_1(z),F_2(z),\ldots,F_m(z))$, we consider the function $g_i: \mathbb{D}\to\ol{\mathbb{D}}$ defined by 
\beas g_i(t):=F_i(t z_0)=a_i+\sum_{n=1}^\infty (P_{n})_i(z_0)t^n,\quad t\in\mathbb{D}\quad \text{and}\quad i\in\{1,2,\ldots,m\}.\eeas
In view of \textrm{Lemmas \ref{lem1}} and \ref{2lem2}, for each fixed $i\in\{1,2,\ldots,m\}$, we have
\bea\label{c2} &&|(P_{n})_i(z_0)|\leq 1-|a_i|^2=1-\Vert a\Vert^2_{\infty}\quad \text{for}\quad n\geq 1\quad\text{and}\\[1mm]
\label{c3} &&\sum_{n=1}^\infty \left|(P_{n})_i(z_0)\right| |t|^{n}+\left(\frac{1}{1+|a_i|}+\frac{|t|}{1-|t|}\right) \sum_{n=1}^\infty \left|(P_{n})_i(z_0)\right|^2 |t|^{2n}\nonumber\hspace{2cm}\\
&\leq& (1-|a_i|^2)\frac{|t|}{1-|t|}=(1-\Vert a\Vert_\infty^2)\frac{|t|}{1-|t|}.\eea
Thus, we have
\bea\label{c4}  \Vert P_n(z_0)\Vert_{\infty}=\max\limits_{1\leq i\leq m} |(P_{n})_i(z_0)|\leq 1-\Vert a\Vert^2_{\infty}\quad \text{for}\quad n\geq 1.\eea
Let $k\in\{1,2,\ldots,m\}$ be such that $|(P_{n})_k(z_0)|=\Vert P_n(z_0)\Vert_\infty=\max_{1\leq i\leq m} \left|(P_n)_i(z_0)\right|$. By setting $|t|=r<1$ in (\ref{c3}), we have 
\bea\label{c5}&&\sum_{n=1}^\infty \left\Vert(P_{n})(z_0)\right\Vert_\infty r^{n}+\left(\frac{1}{1+\Vert a\Vert_\infty}+\frac{r}{1-r}\right) \sum_{n=1}^\infty \left\Vert(P_{n})(z_0)\right\Vert_\infty^2 r^{2n}\nonumber\\[2mm]
&=&\sum_{n=1}^\infty \left(\max_{1\le i\leq m}\left|(P_{n})_i(z_0)\right|\right) r^{n}+\left(\frac{1}{1+\Vert a\Vert_\infty}+\frac{r}{1-r}\right) \sum_{n=1}^\infty \left(\max_{1\leq i\leq m}\left|(P_{n})_i(z_0)\right|\right)^2 r^{2n}\nonumber\eea
\bea &=&\sum_{n=1}^\infty \left|(P_{n})_k(z_0)\right| r^{n}+\left(\frac{1}{1+\Vert a\Vert_\infty}+\frac{r}{1-r}\right) \sum_{n=1}^\infty \left|(P_{n})_k(z_0)\right|^2 r^{2n}\nonumber\hspace{3cm}\\[2mm]
&\leq&(1-\Vert a\Vert_\infty^2)\frac{r}{1-r}.\eea
It is evident that 
\be\label{c6}\Vert F(z)-F(0)\Vert_\infty=\left\Vert \sum_{n=1}^\infty P_n(z)\right\Vert_\infty \leq \sum_{n=1}^\infty \left\Vert P_n(z)\right\Vert_\infty=\sum_{n=1}^\infty\left( \max_{1\leq i\leq m}\left|(P_n)_i(z)\right|\right).\ee
In view of (\ref{c2}) and (\ref{c6}), for $r\in[0,1)$, we have
\bea\label{c7}\Vert F(r z_0)-F(0)\Vert_\infty\leq \sum_{n=1}^\infty \left(\max_{1\leq i\leq m}\left|(P_n)_i(r z_0)\right|\right)&=&\sum_{n=1}^\infty \left|(P_n)_k(z_0)\right| r^n\nonumber\\
&\leq &\left( 1-\Vert a\Vert^2_\infty\right)\sum_{n=1}^\infty r^n\nonumber\\[2mm]
&=&(1-\Vert a\Vert_\infty^2)\frac{r}{1-r}.\eea
Let $x=\Vert a\Vert_{\infty}\in[0,1]$. If $x=1$, then $P_n(z_0)=0$ for $n\geq 1$. Therefore, the inequality (\ref{c1}) holds for $z=rz_0$, where $r\in[0,1)$. Let $x\in[0,1)$. Using (\ref{c4}), (\ref{c5}) and (\ref{c7}), for $r\in(0,1)$, we have
\beas 
&&\Vert F(0)\Vert_\infty^p+\sum_{n=1}^\infty \Vert P_n(r z_0)\Vert_\infty+\left(\frac{1}{1+\Vert a\Vert_\infty}+\frac{\Vert r z_0\Vert_X}{1-\Vert r z_0\Vert_X}\right)\sum_{n=1}^\infty \Vert P_n(r z_0)\Vert_\infty^2\\[2mm]
&&+\Vert F(r z_0)-F(0)\Vert_\infty\\[2mm]
&=&x^p+\sum_{n=1}^\infty \Vert P_n(z_0)\Vert_\infty r^n+\left(\frac{1}{1+x}+\frac{r}{1-r}\right)\sum_{n=1}^\infty \Vert P_n( z_0)\Vert_\infty^2 r^{2n}\\[2mm]
&&+\Vert F(r z_0)-F(0)\Vert_\infty\\[2mm]
&\leq&1-(1-x^p)+2(1-x^2)\frac{r}{1-r}=
\left\{\begin{array}{lll}
1+(1-x)G_3(x,r)&\text{for}~p=1,\\[2mm]
1+(1-x^2)G_4(r)&\text{for}~p=2,\end{array}\right.\eeas
where
\beas G_3(x,r)=2(1+x)\frac{r}{1-r}-1\quad\text{and}\quad G_4(r)=\frac{3r-1}{1-r},\quad x\in[0,1)~\text{and}~r\in(0,1). \eeas
It is evident that $G_3(x,r)$ is a monotonically increasing function of $x\in[0,1)$ and hence, we have  
\beas G_3(x,r)\leq \lim_{x\to1^-}G_3(x,r)=\frac{4r}{1-r}-1\leq 0\quad \text{for}\quad r\leq r_1=\frac{1}{5}.\eeas
Furthermore, it is clear that $G_4(r)\leq 0$ for $r\leq r_2=1/3$.\\[2mm]
\indent To prove the sharpness of the constant $r_p$ for $p\in\{1,2\}$, we consider the function $F_2:\mathbb{B}_X\to \ol{\mathbb{D}^m}$ defined by
\beas F_2(z)=g(l_{z_0}(z))u\quad\text{for}\quad z\in \Bbb{B}_X,\eeas
where $z_0\in \pa \Bbb{B}_X$, $l_{z_0}\in T(z_0)$, $u=(u_1,u_2,\dots,u_m)\in \pa \mathbb{D}^m$ with $|u_1|=|u_2|=\dots =|u_m|=1$, and 
\beas g(t)= \frac{\lambda -t}{1-\lambda t}, \quad t\in \mathbb{D}\quad\text{and}\quad\lambda \in (0,1).\eeas
Therefore, for $r\in (0, 1)$, we have 
\beas F_2(rz_0)=g(l_{z_0}(rz_0))u =g\left(r\Vert z_0\Vert_X\right)u=\frac{\lambda -r}{1-\lambda r}u= a+\sum_{n=1}^\infty P_n(rz_0),\eeas
where $a=\lambda u$ and $P_n(z_0)=\lambda^{n-1}(\lambda^2 -1)u$ for $n\geq 1$. Thus, we have
\beas &&\Vert F_2(0)\Vert_\infty^p+\sum_{n=1}^\infty \Vert P_n(r z_0)\Vert_\infty+\left(\frac{1}{1+\Vert a\Vert_\infty}+\frac{\Vert r z_0\Vert_X}{1-\Vert r z_0\Vert_X}\right)\sum_{n=1}^\infty \Vert P_n(r z_0)\Vert_\infty^2\\[2mm]
&&+\Vert F(r z_0)-F(0)\Vert_\infty\\[2mm]
&=&\Vert \lambda u\Vert_\infty^p+(1-\lambda^2 )\sum_{n=1}^\infty \lambda^{n-1} r^n+\left(\frac{1}{1+\Vert \lambda u\Vert_\infty}+\frac{r}{1-r}\right)(1-\lambda^2)^2\sum_{n=1}^\infty \lambda^{2(n-1)}r^{2n}\\[2mm]
&&+\left\Vert\left( \frac{\lambda -r}{1-\lambda r}-\lambda\right)u\right\Vert_\infty\eeas
\beas&=&\frac{(1-\lambda^2)r}{1-\lambda r}+\lambda^p+(1-\lambda^2)\frac{r}{1-\lambda r}+\left(\frac{1}{1+\lambda}+\frac{r}{1-r}\right)(1-\lambda^2)^2\frac{r^2}{1-\lambda^2 r^2}\\[2mm]
&=&\frac{(1-\lambda^2)r}{1-\lambda r}+\lambda^p+(1-\lambda^2)\frac{r}{1-\lambda r}+\frac{1+\lambda r}{(1+\lambda)(1-r)}\frac{(1-\lambda^2)^2r^2}{1-\lambda^2 r^2}\\[2mm]
&=&\frac{(1-\lambda^2)r}{1-\lambda r}+\lambda^p+(1-\lambda^2)\frac{r}{1-\lambda r}\left(1+\frac{(1-\lambda) r}{1-r}\right)\\[2mm]
&=&\frac{(1-\lambda^2)r}{1-\lambda r}+\lambda^p+(1-\lambda^2)\frac{r}{1-r}=
\left\{\begin{array}{lll}
1+(1-\lambda)G_5(\lambda, r)&\text{for}~p=1,\\[2mm]
1+(1-\lambda^2)G_6(\lambda, r)&\text{for}~p=2,\end{array}\right.\eeas
where
\beas G_5(\lambda, r)=\frac{(1+\lambda)r}{1-\lambda r}+(1+\lambda)\frac{r}{1-r}-1\quad\text{and}\quad G_6(\lambda, r)=\frac{r}{1-\lambda r}+\frac{r}{1-r}-1.\eeas
It is evident that 
\beas &&\lim_{\lambda\to1^-}G_5(\lambda, r)=\frac{5r-1}{1-r}>0\quad\text{for}\quad r>r_1=\frac{1}{5}\\[2mm]\text{and}
&&\lim_{\lambda\to1^-}G_6(\lambda, r)=\frac{3r-1}{1-r}>0\quad\text{for}\quad r>r_2=\frac{1}{3},\eeas
which shows that the constant $r_p$ $(p=1, 2)$ cannot be improved. This completes the proof.
\end{proof}
\noindent The condition $\Vert a\Vert_\infty=|a_j|$ for each $j\in\{1,2,\ldots, m\}$ cannot be omitted from \textrm{Theorem \ref{Th2}}. To illustrate the importance of this condition, we present the following example.
\begin{exam} Consider the holomorphic mapping $F:\mathbb{B}_X\to\ol{\mathbb{D}^2}$ with
\[F(z)=\left(\frac{a_1-l_{z_0}(z)}{1-a_1 l_{z_0}(z)}, \frac{a_2-l_{z_0}(z)}{1-a_2 l_{z_0}(z)}\right),~z\in\mathbb{B}_X, \] 
where $1/\sqrt{2}<a_1<a_2<1$, $z_0\in\pa \mathbb{B}_X$ and $l_{z_0}\in T(z_0)$. Using similar argument as in \textrm{Example \ref{Exam1}}, for $r\in(0, 1)$, we have
\beas F(r z_0)=a+\sum_{n=1}^\infty P_n(r z_0),\eeas
where $a=(a_1, a_2)$ and $P_n(r z_0)=\left((a_1^2-1)a_1^{n-1}r^n, (a_2^2-1)a_2^{n-1}r^n\right)$ for $n\geq 1$. Thus, we have $\Vert P_n(r z_0)\Vert_\infty=(1-a_1^2)a_1^{n-1}r^n$ for $n=1,2, 3$
and $\Vert P_n(r z_0)\Vert_\infty\geq (1-a_2^2)a_2^{n-1}r^n$ for $n\geq 4$. It is evident that 
\beas  F(r z_0)-F(0)=\left(\frac{a_1-r}{1-a_1 r}-a_1, \frac{a_2-r}{1-a_2 r}-a_2\right)=\left(\frac{(a_1^2-1)r}{1-a_1 r}, \frac{(a_2^2-1)r}{1-a_2 r}\right),\eeas 
and hence, we have
\beas \Vert F(r z_0)-F(0)\Vert_\infty=\max_{1\leq i\leq 2}\left\{\left|\frac{(a_i^2-1)r}{1-a_i r}\right|\right\}\geq0\quad\text{and}\quad \Vert F(0)\Vert_\infty=\Vert a\Vert_\infty=a_2.\eeas
Thus, we have
\beas &&S(r z_0):=\Vert F(r z_0)-F(0)\Vert_\infty+\Vert F(0)\Vert_\infty^p+\sum_{n=1}^\infty \Vert P_n(r z_0)\Vert_\infty\\
&&+\left(\frac{1}{1+\Vert a\Vert_\infty}+\frac{\Vert r z_0\Vert_X}{1-\Vert r z_0\Vert_X}\right)\sum_{n=1}^\infty \Vert P_n(r z_0)\Vert_\infty^2\\ 
&=&\max_{1\leq i\leq 2}\left\{\frac{(1-a_i^2)r}{1-a_i r}\right\}+a_2^p+(1-a_1^2)r(1+a_1r+a_1^2 r^2)+\sum_{n=4}^\infty \Vert P_n(r z_0)\Vert_\infty\\[1mm]
&&+\left(\frac{1}{1+a_2}+\frac{r}{1-r}\right)\left((1-a_1^2)^2r^2(1+a_1^2r^2+a_1^4 r^4)+\sum_{n=4}^\infty \Vert P_n(z_0)\Vert_\infty^2 r^{2n}\right)\\
&\geq &\max_{1\leq i\leq 2}\left\{\frac{(1-a_i^2)r}{1-a_i r}\right\}+a_2^p+(1-a_1^2)r(1+a_1r+a_1^2 r^2)+(1-a_2^2)\sum_{n=4}^\infty a_2^{n-1}r^{n}\nonumber\\
&&+\frac{1+a_2 r}{(1-r)(1+a_2)}\left((1-a_1^2)^2r^2(1+a_1^2r^2+a_1^4 r^4)+(1-a_2^2)^2\sum_{n=4}^\infty a_2^{2(n-1)} r^{2n}\right)\\
&\geq &a_2^p+(1-a_1^2)r(1+a_1r+a_1^2 r^2)+(1-a_2^2)\frac{a_2^3 r^4}{1-a_2 r}\nonumber\\
&&+\frac{1+a_2 r}{(1-r)(1+a_2)}\left((1-a_1^2)^2r^2(1+a_1^2r^2+a_1^4 r^4)+(1-a_2^2)^2 \frac{a_2^6r^8}{1-a^2 r^2}\right).\eeas
Thus, for each $r\in(0, 1)$ and $a_2\to 1^-$, we have  
\beas S(r z_0)\geq 1+(1-a_1^2)r(1+a_1r+a_1^2 r^2)+\frac{(1+r)(1-a_1^2)^2r^2(1+a_1^2r^2+a_1^4 r^4)}{2(1-r)}>1.\eeas
\end{exam}
In the following result, we establish a refined version of the Bohr inequality for holomorphic mapping $F:\mathbb{B}_X\to\ol{\mathbb{D}^m}$, where $\Bbb{B}_X$ denote the unit ball of a complex Banach space $X$.
\begin{theo}\label{Th3}
Let $a=(a_1, a_2, \ldots, a_m)\in\ol{\mathbb{D}^m}$ and let $F: \mathbb{B}_X\to \ol{\mathbb{D}^m}$ be a holomorphic mapping with the following expansion
\beas F(z)=a+\sum_{n=1}^\infty P_n(z),\quad z\in \Bbb{B}_X,\eeas
where $P_n(z)=(1/n!) D^n F(0)(z^n)$. Assume that $|a_i|=\Vert a\Vert_\infty$ for all $i\in\{1, 2, \ldots,m\}$ and $\nu: \Bbb{B}_X\to \Bbb{B}_X$ is a Schwarz mapping with $\nu(0)=D \nu(0)=\cdots=D^{k-1} \nu(0)=0$, where $k\in\mathbb{N}$. Then, we have 
\bea\label{d1} \Vert F(\nu(z))\Vert_\infty+\sum_{n=1}^\infty \Vert P_n(z)\Vert_\infty+\left(\frac{1}{1+\Vert a\Vert_\infty}+\frac{\Vert z\Vert_X}{1-\Vert z\Vert_X}\right)\sum_{n=1}^\infty \Vert P_n(z)\Vert_\infty^2\leq 1\eea
for $\Vert z\Vert_X=r\leq r_k$, where $r_k\in(0, 1)$ is the positive root of the equation 
\beas \frac{1-r^{k}}{1+r^k}-\frac{2r}{1-r}= 0.\eeas
The number $r_k$ cannot be improved.
 \end{theo}
\begin{proof}
Let $z_0\in\pa \mathbb{B}_X$ be fixed. As $F(z)=(F_1(z),F_2(z),\ldots,F_m(z))$, we consider the function $g_i: \mathbb{D}\to\ol{\mathbb{D}}$ defined by 
\beas g_i(t):=F_i(t z_0)=a_i+\sum_{n=1}^\infty (P_{n})_i(z_0)t^n,\quad t\in\mathbb{D}\quad \text{and}\quad i\in\{1,2,\ldots,m\}.\eeas
In view of \textrm{Lemmas \ref{lem1}} and \ref{2lem2}, we have $|(P_{n})_i(z_0)|\leq 1-|a_i|^2=1-\Vert a\Vert^2_{\infty}$ for $n\geq 1$ and
\beas\sum_{n=1}^\infty \left|(P_{n})_i(z_0)\right| |t|^{n}+\left(\frac{1}{1+|a_i|}+\frac{|t|}{1-|t|}\right) \sum_{n=1}^\infty \left|(P_{n})_i(z_0)\right|^2 |t|^{2n}
&\leq& (1-|a_i|^2)\frac{|t|}{1-|t|}\\&=&(1-\Vert a\Vert_\infty^2)\frac{|t|}{1-|t|}\eeas
for each fixed $i\in\{1,2,\ldots,m\}$. Using the argument as in the proof of \textrm{Theorem \ref{Th2}}, we obtain the inequalities (\ref{c4}) and (\ref{c5}).
Let us consider the function $G(r)=(b+r)/(1+b r)$, where $b\geq 0$ and $0\leq r\leq r_0 (\leq1)$. Evidently, $G(r)$ is a monotonically increasing function of $r$ 
and hence, we have $G(r)\leq G(r_0)$. In view of \textrm{Lemmas \ref{lem3}} and \ref{lem4}, we have 
\bea\label{d2} \Vert F(\nu(r z_0))\Vert_\infty \leq \frac{\Vert F(0)\Vert_\infty +\Vert \nu(r z_0)\Vert_X}{1+\Vert F(0)\Vert_\infty \Vert \nu(r z_0)\Vert_X}\leq \frac{\Vert a\Vert_\infty +r^k}{1+\Vert a\Vert_\infty r^k},\quad r\in[0, 1). \eea 
Let $x=\Vert a\Vert_{\infty}\in[0,1]$. If $x=1$, then $P_n(z_0)=0$ for all $n\geq 1$ and hence, the inequality (\ref{d1}) holds for $z=rz_0$ and $r\in[0,1)$. Let $x\in[0,1)$. Using (\ref{c4}), (\ref{c5}) and (\ref{d2}), for $r\in(0,1)$, we obtain
\beas&&\Vert F(\nu(r z_0))\Vert_\infty+\sum_{n=1}^\infty \Vert P_n(r z_0)\Vert_\infty+\left(\frac{1}{1+\Vert a\Vert_\infty}+\frac{\Vert r z_0\Vert_X}{1-\Vert r z_0\Vert_X}\right)\sum_{n=1}^\infty \Vert P_n(z)\Vert_\infty^2\\
&\leq & \frac{x +r^k}{1+x r^k}+(1-x^2)\frac{r}{1-r}:=G_7(x, r).\eeas
To establish our result, it suffices to demonstrate that $G_7(x,r)\leq 1$.
Differentiating partially $G_7(x,r)$ twice with respect to $x$, we obtain
\beas \frac{\pa}{\pa x}G_7(x,r)=\frac{1-r^{2 k}}{\left(1+x r^k\right)^2}-\frac{2x r}{1-r}\quad\text{and}\quad \frac{\pa^2}{\pa x^2}G_7(x,r)=-\frac{2 r^k \left(1-r^{2 k}\right)}{\left(1+x r^k\right)^3}-\frac{2r}{1-r}\leq 0.\eeas
Thus, $\frac{\pa}{\pa x}G_7(x,r)$ is a monotonically decreasing function of $x\in[0, 1)$ and it follows that
\beas \frac{\pa}{\pa x}G_7(x,r)\geq \lim_{x\to1^-}\frac{\pa}{\pa x}G_2(x,r)=\frac{1-r^{k}}{1+r^k}-\frac{2r}{1-r}\geq 0,\eeas
for $r\leq r_k$, where $r_k\in(0, 1)$ is the positive root of the equation
\beas G_8(r):=\frac{1-r^{k}}{1+r^k}-\frac{2r}{1-r}= 0.\eeas
It is evident that $G_8(0)=1$ and $\lim_{r\to1^-} G_8(r)=-\infty$.
Therefore, $G_7(x,r)$ is a monotonically increasing function of $x\in [0, 1)$ for $r\leq r_k$ and it follows that
\beas G_7(x,r) \leq \lim_{x\to1^-}G_7(x,r)=1 \quad \text{for}\quad r\leq r_k.\eeas
\indent To prove the sharpness of the constant $r_k$, we consider the function $F_3: \mathbb{B}_X\to \ol{\mathbb{D}^m}$ defined by
\beas F_3(z)=g(l_{z_0}(z))u\quad \text{for}\quad z\in \Bbb{B}_X,\eeas
where $z_0\in \pa \Bbb{B}_X$, $l_{z_0}\in T(z_0)$, $u=(u_1,u_2,\dots,u_m)\in \pa \mathbb{D}^m$ with $|u_1|=|u_2|=\dots =|u_m|=1$, and 
\beas g(t)= \frac{\lambda+t}{1+\lambda t}, \quad t\in \mathbb{D}\quad\text{and}\quad \lambda \in (0,1).\eeas
Let $\nu(z)=z\left( l_{z_0}(z)\right)^{k-1}$.
Thus, we have 
\beas F_3(rz_0)=g(l_{z_0}(rz_0))u =g\left(r\Vert z_0\Vert_X\right)u=\frac{\lambda +r}{1+\lambda r}u= a+\sum_{n=1}^\infty P_n(rz_0),\eeas
where $a=\lambda u$ and $P_n(z_0)=(-1)^n\lambda^{n-1}(\lambda^2 -1)u$ for $n\geq 1$. 
Furthermore,
\beas F_3(\nu(rz_0))=F_3\left(rz_0 \left( l_{z_0}(r z_0)\right)^{k-1}\right)=F_3\left(r^k z_0\left(\Vert z_0\Vert_X\right)^{k-1}\right)=g(l_{z_0}(r^k z_0))u =\frac{\lambda +r^k}{1+\lambda r^k}u.\eeas
Therefore, we have
\beas &&\Vert F(\nu(r z_0))\Vert_\infty+\sum_{n=1}^\infty \Vert P_n(r z_0)\Vert_\infty+\left(\frac{1}{1+\Vert a\Vert_\infty}+\frac{\Vert r z_0\Vert_X}{1-\Vert r z_0\Vert_X}\right)\sum_{n=1}^\infty \Vert P_n(r z_0)\Vert_\infty^2\\[2mm]
&=&\left\Vert\frac{\lambda +r^k}{1+\lambda r^k}u\right\Vert_\infty+(1-\lambda^2 )\sum_{n=1}^\infty \lambda^{n-1} r^n+\left(\frac{1}{1+\Vert \lambda u\Vert_\infty}+\frac{r}{1-r}\right)(1-\lambda^2)^2\sum_{n=1}^\infty \lambda^{2(n-1)}r^{2n}\\[1mm]
&=&\frac{\lambda +r^k}{1+\lambda r^k}+(1-\lambda^2)\frac{r}{1-\lambda r}+\left(\frac{1}{1+\lambda}+\frac{r}{1-r}\right)(1-\lambda^2)^2\frac{r^2}{1-\lambda^2 r^2}\eeas
\beas&=&\frac{\lambda +r^k}{1+\lambda r^k}+(1-\lambda^2)\frac{r}{1-\lambda r}+\frac{1+\lambda r}{(1+\lambda)(1-r)}\frac{(1-\lambda^2)^2r^2}{1-\lambda^2 r^2}\hspace{4cm}\\[2mm]
&=&1+\left(\frac{\lambda +r^k}{1+\lambda r^k}-1\right)+(1-\lambda^2)\frac{r}{1-\lambda r}\left(1+\frac{(1-\lambda) r}{1-r}\right)\\[2mm]
&=&1+\left(\frac{\lambda +r^k}{1+\lambda r^k}-1\right)+(1-\lambda^2)\frac{r}{1-r}\\[2mm]
&=&1+(1-\lambda)G_9(\lambda, r)\eeas
where
\beas G_9(\lambda, r)=\frac{1}{1-\lambda}\left(\frac{\lambda +r^k}{1+\lambda r^k}-1\right)+(1+\lambda)\frac{r}{1-r}.\eeas
It is evident that 
\beas \lim_{\lambda\to1^-}G_9(\lambda, r)&=&\lim_{\lambda\to1^-}\frac{1}{1-\lambda}\left(\frac{\lambda +r^k}{1+\lambda r^k}-1\right)+2\frac{r}{1-r}\\[2mm]
&=&-\lim_{\lambda\to 1^-}\frac{1-r^{2 k}}{\left(1+\lambda r^k\right)^2}+\frac{2r}{1-r}=-\frac{1-r^{k}}{1+r^k}+\frac{2r}{1-r}>0\quad\text{for}\quad r>r_k,
\eeas
which shows that the constant $r_k$ cannot be improved. This completes the proof.
\end{proof}
\noindent The condition $\Vert a\Vert_\infty=|a_j|$ for each $j\in\{1,2,\ldots, m\}$ is crucial for \textrm{Theorem \ref{Th3}} and cannot be omitted. To illustrate the importance of this condition, we present the following example.
\begin{exam} Consider the holomorphic mapping $F:\mathbb{B}_X\to\ol{\mathbb{D}^2}$ with
\[F(z)=\left(\frac{a_1+l_{z_0}(z)}{1+a_1 l_{z_0}(z)}, \frac{a_2+l_{z_0}(z)}{1+a_2 l_{z_0}(z)}\right),~z\in\mathbb{B}_X, \] 
where $0<a_1<a_2<1$, $z_0\in\pa \mathbb{B}_X$ and $l_{z_0}\in T(z_0)$. Let $\nu(z)=z\left( l_{z_0}(z)\right)^{k-1}$. Using the argument as in \textrm{Example \ref{Exam1}}, for $r\in(0, 1)$, we have
\beas F(r z_0)=a+\sum_{n=1}^\infty P_n(r z_0),\eeas
where $a=(a_1, a_2)$ and $P_n(r z_0)=\left((-1)^n (a_1^2-1)a_1^{n-1}r^n, (-1)^n(a_2^2-1)a_2^{n-1}r^n\right)$ for $n\geq 1$. Thus, we have $\Vert P_1(r z_0)\Vert_\infty=(1-a_1^2)r$ and $\Vert P_n(r z_0)\Vert_\infty\geq (1-a_2^2)a_2^{n-1}r^n$ for $n\geq 2$.
Therefore,
\beas &&\Vert F(\nu(r z_0))\Vert_\infty+\sum_{n=1}^\infty \Vert P_n(r z_0)\Vert_\infty+\left(\frac{1}{1+\Vert a\Vert_\infty}+\frac{\Vert r z_0\Vert_X}{1-\Vert r z_0\Vert_X}\right)\sum_{n=1}^\infty \Vert P_n(r z_0)\Vert_\infty^2\\
&=&\frac{a_2 +r^k}{1+a_2 r^k}+(1-a_1^2)r+\sum_{n=2}^\infty \Vert P_n(r z_0)\Vert_\infty+\left(\frac{1}{1+a_2}+\frac{r}{1-r}\right)(1-a_1^2)^2r^2\\
&&+\left(\frac{1}{1+a_2}+\frac{r}{1-r}\right)\sum_{n=2}^\infty \Vert P_n(z_0)\Vert_\infty^2 r^{2n}\eeas
\beas&\geq &\left(\frac{a_2 +r}{1+a_2 r}\right)^2+(1-a_1^2) r+(1-a_2^2)\sum_{n=2}^\infty a_2^{n-1}r^{n}+\frac{1+a_2 r}{(1-r)(1+a_2)}(1-a_1^2)^2r^2\nonumber\\
&&+\frac{1+a_2 r}{(1-r)(1+a_2)}(1-a_2^2)^2\sum_{n=2}^\infty a_2^{2(n-1)} r^{2n}\\
&=&\left(\frac{a_2 +r}{1+a_2 r}\right)^2+(1-a_1^2)r+(1-a_2^2)\frac{a_2 r^2}{1-a_2r}+\frac{1+a_2 r}{(1-r)(1+a_2)}(1-a_1^2)^2r^2\\
&&+\frac{1+a_2 r}{(1-r)(1+a_2)}(1-a_2^2)^2\frac{a_2^2 r^4}{1-a_2^2 r^2}.\eeas
Thus, for each $r\in(0, 1)$ and $a_2\to 1^-$, we have  
\beas&& \Vert F(\nu(r z_0))\Vert_\infty+\sum_{n=1}^\infty \Vert P_n(r z_0)\Vert_\infty+\left(\frac{1}{1+\Vert a\Vert_\infty}+\frac{\Vert r z_0\Vert_X}{1-\Vert r z_0\Vert_X}\right)\sum_{n=1}^\infty \Vert P_n(r z_0)\Vert_\infty^2\\
&\geq& 1+(1-a_1^2)r+\frac{1+a_2 r}{(1-r)(1+a_2)}(1-a_1^2)^2r^2>1.\eeas
\end{exam}
\section*{Declarations:}
\noindent{\bf Acknowledgment:} First author is supported by the University Grants Commission (IN) fellowship (No. F. 44 - 1/2018 (SA - III)).\\[1mm]
{\bf Conflict of Interest:} Authors declare that they have no conflict of interest.\\[1mm]
{\bf Availability of data and materials:} Not applicable.\\

\end{document}